\documentclass[reqno,a4paper,twoside]{amsart}
\usepackage{enumitem}
\setenumerate{label=\textnormal{(\arabic*)}}

\allowdisplaybreaks[3]

\usepackage{microtype}

\usepackage{amsmath,amssymb,dsfont,mathrsfs,verbatim,bm,geometry}
\usepackage{mathtools}
\usepackage{array,hhline}
\usepackage[raggedright]{titlesec}
\usepackage{tikz}
\usepackage{float,subfig}
\usepackage{amsrefs}

\DeclareMathAlphabet\mathbfcal{OMS}{cmsy}{b}{n}

\usepackage[raggedright]{titlesec}

\titleformat{\chapter}[display]
{\normalfont\huge\bfseries}{\chaptertitlename\\thechapter}{20pt}{\Huge}
\titleformat{\section}
{\normalfont\Large\bfseries\center}{\thesection}{1em}{}
\titleformat{\subsection}
{\normalfont\large\bfseries}{\thesubsection}{1em}{}
\titleformat{\subsubsection}
{\normalfont\normalsize\bfseries}{\thesubsubsection}{1em}{}
\titleformat{\paragraph}[runin]
{\normalfont\normalsize\bfseries}{\theparagraph}{1em}{}
\titleformat{\subparagraph}[runin]
{\normalfont\normalsize\bfseries}{\thesubparagraph}{1em}{}

\titlespacing*{\chapter} {0pt}{50pt}{40pt}
\titlespacing*{\section} {0pt}{3.5ex plus 1ex minus .2ex}{2.3ex plus .2ex}
\titlespacing*{\subsection} {0pt}{3.25ex plus 1ex minus .2ex}{1.5ex plus .2ex}
\titlespacing*{\subsubsection}{0pt}{3.25ex plus 1ex minus .2ex}{1.5ex plus .2ex}
\titlespacing*{\paragraph} {0pt}{3.25ex plus 1ex minus .2ex}{1em}
\titlespacing*{\subparagraph} {\parindent}{3.25ex plus 1ex minus .2ex}{1em}

\usepackage{tikz}
\usetikzlibrary{matrix}

\usepackage{amsrefs}

\usepackage{titletoc}

\usepackage{stackengine}
\usepackage{scalerel}

\contentsmargin{2.55em}
\dottedcontents{section}[1.5em]{}{2em}{1pc}
\dottedcontents{subsection}[4.35em]{}{2.8em}{1pc}
\dottedcontents{subsubsection}[7.6em]{}{3.2em}{1pc}
\dottedcontents{paragraph}[10.3em]{}{3.2em}{1pc}

\usepackage[columns=3,rule=0pt]{idxlayout}
\setindexprenote{For convenience of the reader we list below almost all symbols used in this work, indicating the page in which each of them is introduced. We only except those that are extremely standard.}

\usepackage{amsmath,amssymb,dsfont,verbatim,bm,geometry,mathrsfs}

\usepackage{mathtools}

\usepackage[pdftex,breaklinks,pdfencoding=auto,psdextra,unicode,naturalnames]{hyperref}

\newtheorem{theorem}{Theorem}[section]
\newtheorem{lemma}[theorem]{Lemma}
\newtheorem{proposition}[theorem]{Proposition}

\theoremstyle{definition}
\newtheorem{definition}[theorem]{Definition}

\newtheorem{example}[theorem]{Example}

\theoremstyle{remark}
\newtheorem{remark}[theorem]{Remark}

\DeclareMathOperator{\ima}{Im}

\newcommand{\hsa}{\hspace{-0.7pt}}
\newcommand{\hsb}{\hspace{-0.5pt}}

\newcommand{\ov}{\overline}
\newcommand{\ot}{\otimes}

\newcommand{\hsm}{\hspace{-0.2pt}}
\newcommand{\xcdot}{\hsm\cdot\hsm}
\newcommand{\xcirc}{\hsm\circ\hsm}

\DeclareMathAlphabet{\mathpzc}{OT1}{pzc}{m}{it}

\begin{document}

\title{Comparison of two notions of weak crossed product}

\author{Jorge A. Guccione}
\address{Departamento de Matem\'atica\\ Facultad de Ciencias Exactas y Naturales-UBA, Pabell\'on~1-Ciudad Universitaria\\ Intendente Guiraldes 2160 (C1428EGA) Buenos Aires, Argentina.}
\address{Instituto de Investigaciones Matem\'aticas ``Luis A. Santal\'o"\\ Facultad de Ciencias Exactas y Natu\-ra\-les-UBA, Pabell\'on~1-Ciudad Universitaria\\ Intendente Guiraldes 2160 (C1428EGA) Buenos Aires, Argentina.}
\email{vander@dm.uba.ar}

\author{Juan J. Guccione}
\address{Departamento de Matem\'atica\\ Facultad de Ciencias Exactas y Naturales-UBA\\ Pabell\'on~1-Ciudad Universitaria\\ Intendente Guiraldes 2160 (C1428EGA) Buenos Aires, Argentina.}
\address{Instituto Argentino de Matem\'atica-CONICET\\ Savedra 15 3er piso\\ (C1083ACA) Buenos Aires, Argentina.}
\email{jjgucci@dm.uba.ar}

\thanks{Jorge A. Guccione and Juan J. Guccione were supported by UBACyT 20020150100153BA (UBA)}


\begin{abstract} We compare the restriction to the context of weak Hopf algebras of the notion of crossed product with a Hopf algebroid introduced in \cite{BB} with the notion of crossed product with a weak Hopf algebra introduced in~\cite{AG}.
\end{abstract}

\maketitle

\tableofcontents

\section*{Introduction}

Two different notions of crossed product of an algebra $A$ with a weak Hopf algebra $H$ have been purposed. The first one is the restriction to the context of weak Hopf algebras of the notion of crossed product with a Hopf algebroid introduced by B\"ohm and Brzezi\'nski in~\cite{BB}. The second one was introduced in~\cite{AG} and studied in a series of papers (see for instance \cites{AFGR1, AFGR2, AFGR3, FGR, Ra}). The aim of this paper is to prove that there exists the last one if and only if there exists the first one and condition~(10) below is fulfilled. Moreover, in this case both constructions are canonically isomorphic. Finally we provide an example which shows that there are crossed products of algebras with weak Hopf algebras in the sense of B\"ohm and Brzezi\'nski that not satisfy condition~(10).

\section{Preliminaries}\label{section: preliminares}

In this section we review the notion of weak Hopf algebra and the notions of crossed product that we want to compare.

\subsection{Weak Hopf algebras}
Weak bialgebras and weak Hopf algebras are generalizations of bialgebras and Hopf algebras, introduced in~\cites{BNS1,BNS2}, in which the axioms about the unit, the counit and the antipode are replaced by weaker properties. Next we give a brief review of the basic properties of these structures.

\smallskip

Let $k$ be a field. A weak {\em bialgebra} is a $k$-vector space $H$, endowed with an algebra structure and a coalgebra structure, such that $\Delta(hl) =\Delta(h)\Delta(l)$ for all $h,l\in H$, and the equalities
\begin{align}
\label{propiedad de 1}
& \Delta^2(1)= 1^{(1)}\ot 1^{(2)} 1^{(1')}\ot 1^{(2')}= 1^{(1)}\ot 1^{(1')} 1^{(2)}\ot 1^{(2')}
\shortintertext{and}
\label{propiedad de epsilon}
& \epsilon(hlm)=\epsilon(hl^{(1)})\epsilon(l^{(2)}m)=\epsilon(hl^{(2)})\epsilon(l^{(1)}m)\quad\text{for all $h,l,m\in H$,}
\end{align}
are fulfilled, where we are using the Sweedler notation for the coproduct, with the summation symbol omitted. A {\em weak bialgebra morphism} is a function $g\colon H\to L$ that is an algebra and a coalgebra map. For each weak bialgebra~$H$, the maps~$\Pi^L\colon H\to H$ and~$\Pi^R\colon H\to H$ defined by
$$
\Pi^L(h)\coloneqq \epsilon(1^{(1)}h)1^{(2)}\quad\text{and}\quad \Pi^R(h)\coloneqq 1^{(1)}\epsilon(h1^{(2)}),
$$
respectively, are idempotent. We set~$H^{\!L}\coloneqq \ima(\Pi^L)$ and~$H^{\!R}\coloneqq \ima(\Pi^R)$. In~\cite{BNS1} it was proven that~$H^{\!L}$ and~$H^{\!R}$ are subalgebras of~$H$.

\smallskip

Let $H$ be a weak bialgebra. An {\em antipode} of $H$ is a map $S\colon H\to H$ (or $S_H$ if necessary to avoid confusion), such that $h^{(1)}S(h^{(2)})=\Pi^L(h)$, $S(h^{(1)})h^{(2)}=\Pi^R(h)$ and $S(h^{(1)})h^{(2)}S(h^{(3)})=S(h)$, for all $h\in H$. As it was shown in~\cite{BNS1}, an antipode $S$, if there exists, is unique. It was also shown in~\cite{BNS1} that $S$ is antimultiplicative, anticomultiplicative and leaves the unit and counit invariant. A {\em weak Hopf algebra} is a weak bialgebra that has an antipode. A {\em morphism of weak Hopf algebras} $g\colon H\to L$ is simply a bialgebra morphism from $H$ to $L$. In~\cite{AFGLV}*{Pro\-position~1.4} it was proven that if $g\colon H\to L$ is a weak Hopf algebra morphism, then $g\xcirc S_H= S_L\xcirc g$. From now one we assume that $H$ is a weak Hopf algebra with bijective antipode.

\subsection[Crossed product $A\#^{\rho}_{\sigma} H$]{Crossed product $\bm{A\#^{\rho}_{\sigma} H}$}\label{subsection: Bhom y Brzezinksi}
The notion of crossed product with a Hopf algebroid was introduced by B\"ohm and Brzezi\'nski in~\cite{BB}. Since weak Hopf algebras provide examples of Hopf algebroids, this gives a notion of crossed product with a weak Hopf algebra. In this subsection we review this construction. For the proofs of the results we refer to~\cite{BB} (see also \cite{LSW}).

\smallskip

Let $A$ be a $k$-algebra, $H$ a weak Hopf algebra and $\rho\colon H\ot A\to A$ a map. We set $h\xcdot a\coloneqq \rho(h\ot a)$.

\begin{definition}\label{medida de H sobre un algebra} We say that $H$ {\em measures} $A$ via $\rho$ and that $\rho$ is a {\em measuring} if
\begin{enumerate}[itemsep=0.7ex, topsep=1.0ex, label=(\arabic*), leftmargin=0.9cm]

\item $h\xcdot 1_A=\Pi^L(h)\xcdot 1_A$,

\item $h\xcdot (aa') = (h^{(1)}\xcdot a)(h^{(2)}\xcdot a')$,

\item $S^{-1}(l)h\xcdot a = (h\xcdot a)(l\xcdot 1_A)$ and $lh\xcdot a = (l\xcdot 1_A)(h\xcdot a)$

\end{enumerate}
for all $h\in H$, $l\in H^{\!L}$ and $a,a'\in A$. The measuring is unital if
\begin{enumerate}[itemsep=0.7ex, topsep=1.0ex, label=(\arabic*), leftmargin=0.9cm, resume]

\item $1\cdot a=a$ for all $a\in A$.

\end{enumerate}

\end{definition}

\noindent By items~(3) and~(4) we have
$$
h\cdot (k\cdot a) = hk\cdot a\qquad\text{for all $h\in H^{\!L}H^{\!R}$, $k\in H$ and $a\in A$.}
$$
In the rest of this subsection

\begin{itemize}[itemsep=0.7ex, topsep=1.0ex, leftmargin=0.9cm, label=-]

\item we assume that $\rho$ is a unital measuring,

\item given a map $\sigma\colon H\ot_{H^{\!R}} H\to A$ we write $\sigma(h, k)$ instead of $\sigma(h\ot_{H^{\!R}} k)$,

\item given a map $\bar{\sigma}\colon H\ot_{H^{\!L}} H\to A$ we write $\bar{\sigma}(h, k)$ instead of $\bar{\sigma}(h\ot_{H^{\!L}} k)$.

\end{itemize}

\begin{definition}\label{cociclo} A map $\sigma\colon H\ot_{H^{\!R}} H\to A$ is a {\em normal cocycle} if
\begin{enumerate}[itemsep=0.7ex, topsep=1.0ex, label=(\arabic*), start=5, leftmargin=0.9cm]

\item $\sigma(lh, k)=(l\xcdot 1_A)\sigma(h, k)$ and $\sigma(S^{-1}(l)h, k)=\sigma(h, k)(l\xcdot 1_A)$,

\item $(h^{(1)}\xcdot (l\xcdot 1_A))\sigma(h^{(2)}, k)=\sigma(h,lk)$,

\item $\sigma(1, h)=\sigma(h, 1)= h\xcdot 1_A$,

\item $(h^{(1)}\xcdot \sigma(k^{(1)}, m^{(1)}))\sigma(h^{(2)}, k^{(2)}m^{(2)})= \sigma(h^{(1)}, k^{(1)})\sigma(h^{(2)}k^{(2)}, m)$,

\end{enumerate}
for all $h,k,m\in H$ and $l\in H^{\!L}$.
\end{definition}

\begin{definition}\label{modulo algebra} A map $\sigma\colon H\ot_{H^{\!R}} H\to A$ satisfy the {\em twisted module condition} if
\begin{enumerate}[itemsep=0.7ex, topsep=1.0ex, label=(\arabic*), start=9, leftmargin=0.9cm]

\item $(h^{(1)}\cdot (k^{(1)}\cdot a))\sigma(h^{(2)}, k^{(2)}) = \sigma(h^{(1)}, k^{(1)})(h^{(2)}k^{(2)}\cdot a)$ for all $h,k\in H$ and $a\in A$.

\end{enumerate}
\end{definition}

\begin{proposition}\label{pese} The following facts hold:

\begin{itemize}[itemsep=0.7ex, topsep=1.0ex, leftmargin=0.9cm]

\item[-] $h\cdot (k\cdot a) = hk\cdot a$ for all $h\in H^{\!L}H^{\!R}$, $k\in H$ and $a\in A$.

\item[-] $\sigma(h,k) = \sigma(h^{(1)},k^{(1)})(h^{(2)}k^{(2)}\xcdot 1_A)$ for all $h,k\in H$.

\end{itemize}

\end{proposition}

\begin{proof} The first item follows easily from items~(3) and~(4). We next prove the second one. By items~(1), (6) and~(9), we have
\begin{align*}
\sigma(h^{(1)},k^{(1)})(h^{(2)}k^{(2)}\xcdot 1_A) &= (h^{(1)}\xcdot (k^{(1)}\xcdot 1_A)\sigma(h^{(2)},k^{(2)})\\
& = (h^{(1)}\xcdot (\Pi^L(k^{(1)})\xcdot 1_A)\sigma(h^{(2)},k^{(2)})\\
& = \sigma (h,\Pi^L(k^{(1)})k^{(2)}).
\end{align*}
Since $\Pi^L(k^{(1)})k^{(2)} = k$ this ends the proof.
\end{proof}

\begin{theorem}\label{teor: producto cruzado} Let $\sigma\colon H\ot_{H^{\!R}} H\to A$ be a map satisfying conditions~(5) and~(6) of Def\-i\-ni\-tion~\ref{cociclo}. Consider $A$ a right $H^{\!L}$-module via $a\xcdot l\coloneqq a(l\cdot 1_A)$. The $k$-vector space $A\ot_{H^{\!L}} H$ is an associative algebra with unit $1_A\ot_{H^{\!L}} 1$ via
$$
(a\ot_{H^{\!L}} h)(b\ot_{H^{\!L}} k)\coloneqq a(h^{(1)}\cdot b)\sigma(h^{(2)},k^{(1)}) \ot_{H^{\!L}} h^{(3)}k^{(2)}
$$
if and only if the measuring $\rho$ is unital and $\sigma$ is a normal cocycle that satisfies the twisted module condition.
\end{theorem}

\begin{definition}\label{producto cruzado} The algebra constructed in Theorem~\ref{teor: producto cruzado} is denoted $A\#^{\rho}_{\sigma} H$ and is called {\em the weak crossed product of $A$ with $H$ associated with $\rho$ and $\sigma$}.
\end{definition}

We consider $A\#^{\rho}_{\sigma} H$ as a left $A$-module via the natural action and as a right $H$-comodule via the map $\delta\colon A\#^{\rho}_{\sigma} H\to A\#^{\rho}_{\sigma} H\ot H$, defined by  $\delta(a\ot_{H^{\!L}}h)\coloneqq a\ot_{H^{\!L}}h^{(1)} \ot h^{(2)}$.

\begin{proposition}\label{212 y 213} Under conditions~(1), (2), (6) and~(7) the following assertions are equivalent

\begin{enumerate}[itemsep=0.7ex, topsep=1.0ex, label=\emph{(\arabic*)}, start=10, leftmargin=0.9cm]

\item $h\xcdot (l\xcdot 1_A)=hl\xcdot 1_A$ for all $h\in H$ and $l\in H^{\!L}$,

\item $h\xcdot (l\xcdot 1_A)=hl\xcdot 1_A$ for all $h,l\in H$,

\item $\sigma(h,l)=\sigma(hl,1)$ for all $h\in H$ and $l\in H^{\!L}$.

\end{enumerate}

\end{proposition}

\begin{proof} (1) $\Leftrightarrow$ (2).\enspace Since, by item~(1) and \cite{BNS1}*{Equality~(2.5a)}, we have
$$
hl\xcdot 1_A=\Pi^L(hl)\xcdot 1_A=\Pi^L(h\Pi^L(l))\xcdot 1_A=h\Pi^L(l)\xcdot 1_A\quad\text{and}\quad h\xcdot(\Pi^L(l)\xcdot 1_A)= h\xcdot (l\xcdot 1_A),
$$
for all $h,l\in H$.

\smallskip

\noindent (1) $\Leftrightarrow$ (3)\enspace Since, by items~(2), (6) and~(7), we have
$$
\sigma(h,l) = (h^{(1)}\xcdot (l\xcdot 1_A))\sigma(h^{(2)},1) = (h^{(1)}\xcdot (l\xcdot 1_A))(h^{(2)}\xcdot 1_A)=h\xcdot (l\xcdot 1_A)\quad\text{and}\quad hl\xcdot 1_A = \sigma(hl,1).
$$
for all $h$ and $l\in H^L$.
\end{proof}

\begin{definition}\label{cociclo inversible} We say that a normal cocycle $\sigma\colon H\ot_{H^{\!R}} H\to A$ satisfying the twisted module condition is invertible if there exists a map $\bar{\sigma}\colon H\ot_{H^{\!L}} H\to A$ satisfying
\begin{enumerate}[itemsep=0.7ex, topsep=1.0ex, label=(\arabic*), start=13, leftmargin=0.9cm]

\item $\bar{\sigma}(h,k) = (h^{(1)}k^{(1)}\cdot 1_A) \bar{\sigma}(h^{(2)},k^{(2)})$,

\item $\bar{\sigma}(lh,k)=(l\xcdot 1_A)\bar{\sigma}(h,k)$ and $\bar{\sigma}(S^{-1}(l)h,k)=\bar{\sigma}(h,k)(l\xcdot 1_A)$,

\item $\bar{\sigma}(h^{(1)},k) (h^{(2)}\xcdot (l\xcdot 1_A))=\bar{\sigma}(h,S^{-1}(l)k)$,

\item $\sigma(h^{(1)}, k^{(1)})\bar{\sigma}(h^{(2)}, k^{(2)})= h\xcdot (k\xcdot 1_A)$ and $\bar{\sigma}(h^{(1)}, k^{(1)})\sigma(h^{(2)}, k^{(2)})= hk\xcdot 1_A$,

\end{enumerate}
for all $h,k\in H$ and $l\in H^{\!L}$.
\end{definition}

\begin{remark}\label{rere1q} By item~(13) and~(16), we have
\begin{align*}
\bar{\sigma}(h^{(1)},k^{(1)}) (h^{(2)}\cdot (k^{(2)}\cdot 1_A)) & = \bar{\sigma}(h^{(1)}, k^{(1)})\sigma(h^{(2)}, k^{(2)}) \bar{\sigma}(h^{(3)},k^{(3)})\\
& = (h^{(1)}k^{(1)}\cdot 1_A) \bar{\sigma}(h^{(2)},k^{(2)})\\
& = \bar{\sigma}(h,k).
\end{align*}
An standard computation using this shows that $\bar{\sigma}$, which is called the {\em inverse of $\sigma$}, is unique.
\end{remark}

\begin{remark}\label{es mas suficiente} It $\bar{\sigma}\colon H\ot_{H^{\!L}} H\to A$ satisfies items~(14)--(16), then the map $\tilde{\sigma}\colon H\ot_{H^{\!L}} H\to A$, defined by $\tilde{\sigma}(h,k)\coloneqq (h^{(1)}k^{(1)}\cdot 1_A) \bar{\sigma}(h^{(2)},k^{(2)})$, satisfies items~(13)--(16).
\end{remark}

\begin{align*}
\tilde{\sigma}(h,k) & = (h^{(1)}k^{(1)}\cdot 1_A) \bar{\sigma}(h^{(2)},k^{(2)})\\
& = (h^{(1)}k^{(1)}\cdot 1_A) (h^{(2)}k^{(2)}\cdot 1_A) \bar{\sigma}(h^{(3)},k^{(3)})\\
& = (h^{(1)}k^{(1)}\cdot 1_A) \tilde{\sigma}(h^{(2)},k^{(2)}).
\end{align*}

\subsection[Crossed product $A\times^{\rho}_{\varsigma} H$]{Crossed product $\bm{A\times^{\rho}_{\varsigma} H}$}\label{subsection: espanholes}
In this subsection we review the construction of weak crossed products over weak Hopf algebras introduced in~\cite{AG}. As in the previous subsection~$A$ is a $k$-algebra, $H$ is a weak Hopf algebra, $\rho\colon H\ot A\to A$ is a map and we set $h\xcdot a\coloneqq \rho(h\ot a)$.

\begin{definition}\label{def modulo algebra debil} We say that $A$ is a {\em left weak $H$-module algebra} via $\rho$, if it satisfies conditions~(2), (4) and~(11). In this case we say that $\rho$ is a left weak action of $H$ on $A$.
\end{definition}

In the rest of this subsection we assume that $A$ is a left weak $H$-module algebra via $\rho$.

\begin{definition}\label{modulo torcido y de cociclo} Let $\varsigma\colon H\ot H\longrightarrow A$ be a map. We say that:

\begin{enumerate}[itemsep=0.7ex, topsep=1.0ex, label=(\arabic*),start=17, leftmargin=0.9cm]

\item $\varsigma$ {\em is a cocycle} if
$$
\quad\qquad (h^{(1)}\xcdot \varsigma(k^{(1)}\ot m^{(1)})) \varsigma(h^{(2)}\ot k^{(2)}m^{(2)})=\varsigma(h^{(1)}\ot k^{(1)})\varsigma(h^{(2)}k^{(2)}\ot m) \quad\text{for all $h,k,m\in H$.}
$$

\item $\varsigma$ satisfies the {\em twisted module condition} if
$$
\qquad (h^{(1)}\xcdot (k^{(1)}\xcdot a)) \varsigma(h^{(2)}\ot k^{(2)})=\varsigma(h^{(1)}\ot k^{(1)})(h^{(2)}k^{(2)}\xcdot a)\quad\text{for all $h,k\in H$ and $a\in A$.}
$$
\end{enumerate}
\end{definition}

Let $\nu\colon k\to A\ot H$ be the map defined by $\nu(\lambda)\coloneqq \lambda 1^{(1)}\xcdot 1_A\ot 1^{(2)}$. In the sequel we assume that

\begin{enumerate}[itemsep=0.7ex, topsep=1.0ex, label=(\arabic*),start=19, leftmargin=0.9cm]

\item $\varsigma(h\ot k) = \varsigma(h^{(1)}\ot k^{(1)})(h^{(2)}k^{(2)}\xcdot 1_A)$,

\item $h\xcdot 1_A = (h^{(1)}\xcdot (1^{(1)}\xcdot 1_A))\varsigma(h^{(2)} \ot 1^{(2)})$,

\item $h\xcdot 1_A = (1^{(1)}\xcdot 1_A)\varsigma(1^{(2)}\ot h)$,

\item $a(1^{(1)}\xcdot 1_A)\ot 1^{(2)} = 1^{(1)}\xcdot a\ot 1^{(2)}$,

\end{enumerate}
for all $h,k\in H$ and $a\in A$.


\begin{remark}\label{pepe} Under conditions~(18) and~(19), conditions~(20) and~(21) are equivalent to the fact that $\varsigma(1\ot h)=\varsigma(h\ot 1)=h\xcdot 1_A$ for all $h\in H$\footnote{A map $\varsigma$ satisfying this condition is called \emph{normal}.} (see \cite{GGV}*{Remarks~2.17 and~2.19}).
\end{remark}

\begin{remark}\label{pepe'} From condition~(11), (18) and~(19) it follows that
$$
(h^{(1)}k^{(1)}\xcdot 1_A)\varsigma(h^{(2)}\ot k^{(2)})\! =\! (h^{(1)}\xcdot (k^{(1)}\xcdot 1_A))\varsigma(h^{(2)}\ot k^{(2)})\!=\! \varsigma(h^{(1)}\ot k^{(1)})(h^{(2)}k^{(2)}\xcdot 1_A) = \varsigma(h\ot k).
$$
for all $h,k\in H$
\end{remark}

Let $A\times H$ be the image of the map $\nabla_{\!\rho} \colon A\ot H\to A\ot H$, defined by $\nabla_{\!\rho}(a\ot h)\coloneqq a(h^{(1)}\xcdot 1_A)\ot h^{(2)}$. From now on, for each simple tensor $a\ot h\in A\ot H$, we set $a\times h\coloneqq \nabla_{\!\rho}(a\ot h)$.

\begin{theorem}\label{teor: producto cruzado espanhol} Under conditions~(2), (4), (11) and~(17)--(25), the $k$-vector space $A\times H$ is an associative algebra with unit $1_A\times 1$, via
$$
(a\times h)(b\times k)\coloneqq a (h^{(1)}\cdot b)\varsigma(h^{(2)}\ot k^{(1)}) \times h^{(3)}k^{(2)}.
$$
\end{theorem}

\begin{definition}\label{def: producto cruzado espanhol}
The algebra constructed in Theorem~\ref{teor: producto cruzado  espanhol} is denoted $A\times^{\rho}_{\varsigma} H$ and is called {\em the weak crossed product of $A$ with $H$ associated with $\rho$ and $\varsigma$}.
\end{definition}

We consider $A\times^{\rho}_{\varsigma} H$ as a left $A$-module via the natural action and as a right $H$-comodule via the map $\delta\colon A\times^{\rho}_{\varsigma} H\to A\times^{\rho}_{\varsigma} H\ot H$, defined by $\delta(a\times h)\coloneqq (a\times h^{(1)}) \ot h^{(2)}$.

\begin{definition}\label{cociclo inversible espanhol} Assume that conditions~(2), (4), (11) and~(17)--(25) are fulfilled. We say that $\varsigma\colon H\ot H\to A$ is {\em invertible} if there exists a map $\bar{\varsigma}\colon H\ot H\to A$ satisfying
\begin{enumerate}[itemsep=0.7ex, topsep=1.0ex, label=(\arabic*), start=23, leftmargin=0.9cm]

\item $\bar{\varsigma}(h,k) = (h^{(1)}k^{(1)}\cdot 1_A) \bar{\varsigma}(h^{(2)}\ot k^{(2)}) $,


\item $\varsigma(h^{(1)}\ot k^{(1)})\bar{\varsigma}(h^{(2)}\ot k^{(2)})= \bar{\varsigma}(h^{(1)}\ot k^{(1)})\varsigma(h^{(2)}\ot k^{(2)})= hk\xcdot 1_A$,

\end{enumerate}
for all $h,k\in H$.
\end{definition}

The map $\bar{\varsigma}$ is unique, satisfies $\bar{\varsigma}(h \ot k)= (h^{(1)}k^{(1)}\xcdot 1_A)\bar{\varsigma}(h^{(2)}\ot k^{(2)})$ and is called the {\em inverse of $\varsigma$} (see the discussion below \cite{GGV}*{Definition~5.6}).

\section{The comparison}

Let $A$ be a $k$-algebra, $H$ be a weak Hopf algebra and $\rho\colon H\ot A\to A$ a map. As in Sub\-sec\-tions~\ref{subsection: Bhom y Brzezinksi} and~\ref{subsection: espanholes} we set $h\xcdot a\coloneqq \rho(h\ot a)$.

\subsection[Crossed products $A\#^{\rho}_{\sigma} H$ that induce crossed products $A\times^{\rho}_{\varsigma} H$]{Crossed products $\bm{A\#^{\rho}_{\sigma} H}$ that induce crossed products $\bm{A\times^{\rho}_{\varsigma} H}$}
The aim of this subsection is to prove the following result:

\begin{theorem}\label{Bhom y Brzezinski son espanholes} Let $\sigma\colon H\ot_{H^{\!R}} H\to A$  be a map and let $\varsigma \coloneqq \sigma\circ p$, where $p\colon H\ot H \to H\ot_{H^{\!R}} H$ is the canonical surjection. If the maps $\rho$ and $\sigma$ satisfy conditions (1)--(10) then $\rho$ and $\varsigma$ satisfy conditions~(2), (4), (11) and~(17)--(22).
\end{theorem}

From here to the end of this subsection we assume that conditions~(1)--(10) are fulfilled.

\begin{proof}[Proof of Theorem~\ref{Bhom y Brzezinski son espanholes}] Items~(2) and~(4) are trivially fulfilled. By Proposition~\ref{212 y 213}, item~(11) is fulfilled. Items~(17) and~(18) follow from items~(8) and~(9) respectively, while item~(19) follows from Proposition~\ref{pese}. By Proposition~\ref{212 y 213} and items~(2), (7) and~(10), we have
$$
(h^{(1)}\xcdot (1^{(1)}\xcdot 1_A))\varsigma(h^{(2)} \ot 1^{(2)})= (h^{(1)}1^{(1)}\xcdot 1_A)\sigma(h^{(2)}1^{(2)}, 1)=(h^{(1)}\xcdot 1_A)(h^{(2)}\xcdot 1_A)= h\xcdot 1_A,
$$
which proves that item~(20) holds. By items~(4), (6) and~(7), we have
$$
(1^{(1)}\xcdot 1_A)\varsigma(1^{(2)}\ot h)= (1^{(1)}\xcdot (1\cdot 1_A))\sigma(1^{(2)},h)=  \sigma(1,h) = h\xcdot 1_A,
$$
which proves item~(21). Finally, in order to check item~(22), it suffices to note that by Equality~\eqref{propiedad de 1} and items~(2) and~(11),
\begin{align*}
(1^{(1)}\xcdot 1_A)(1^{(2)}\xcdot a)\ot 1^{(3)} & = (1^{(1)}\xcdot a)(1^{(2)}\xcdot 1_A)\ot 1^{(3)}\\
& = (1^{(1)}\xcdot a)(1^{(2)}\xcdot (1^{(1')}\xcdot 1_A))\ot 1^{(2')}\\
&= a(1^{(1)}\xcdot 1_A)\ot 1^{(2)},
\end{align*}
as desired.
\end{proof}

\begin{proposition}\label{sigma inversible implica varsigma inversible} Under the hypotheses of Theorem~\ref{Bhom y Brzezinski son espanholes}, if $\sigma$ is invertible, then $\varsigma$ is also.
\end{proposition}

\begin{proof} Let $\bar{\sigma}$ be the inverse of $\sigma$ and let $\tilde{\varsigma}\coloneqq \bar{\sigma}\xcirc q$, where $q\colon H\ot H \to H\ot_{H^{\!L}} H$ is the canonical surjection. From items~(11), (13) and~(16) it follows that~$\tilde{\varsigma}$ satisfies items~(23) and~(24).
\end{proof}

\subsection[Crossed product $A\#^{\rho}_{\sigma} H$ associated with a crossed product  $A\times^{\rho}_{\varsigma} H$]{Crossed product $\bm{A\#^{\rho}_{\sigma} H}$ associated with a crossed product  $\bm{A\times^{\rho}_{\varsigma} H}$}
The aim of this subsection is to prove the following result:

\begin{theorem}\label{espanholes implica Bhom y Brzezinksi} Let $\varsigma\colon H\ot H\to A$  be a map.  If the maps $\rho$ and $\varsigma$ satisfy conditions~~(2), (4), (11) and~(17)--(22), then $\varsigma$ factorizes throughout a map $\sigma\colon H\ot_{H^{\!R}} H\to A$ and the maps $\rho$ and $\sigma$ satisfy conditions~(1)--(9).
\end{theorem}

From here to the end of this subsection we assume that conditions~(2), (4), (11) and~(17)--(22) are fulfilled.

\begin{lemma}\label{condicion'}
The equality $k\xcdot (h\xcdot a)=kh\xcdot a$ holds for all $k\in H^{L}\cup H^{\!R}$, $h\in H$ and $a\in A$.
\end{lemma}

\begin{proof} We have,
\begin{align*}
k\xcdot (h\xcdot a) &= k^{(1)}\xcdot (h^{(1)}\xcdot a)(k^{(2)}h^{(2)}\xcdot 1_A)\\
&= k^{(1)}\xcdot (h^{(1)}\xcdot a )\varsigma (k^{(2)}\ot h^{(2)})\\
&= \varsigma (k^{(1)}\ot h^{(1)})(k^{(2)}h^{(2)}\xcdot a)\\
&= (k^{(1)}h^{(1)}\xcdot 1_A)(k^{(2)}h^{(2)}\xcdot a)\\
&= kh\xcdot a.
\end{align*}
where the first equality follows from items~(2) and~(11); the second and the fourth ones, from Remark~\ref{pepe}, the fact that $\Delta(k)\in H^{\!R}H^{\!L}\ot H^{\!R}H^{\!L}$ and \cite{GGV}*{Propositions~2.7 and~2.8}; the third one, from item~(18); and the last one, from item~(2).
\end{proof}

\begin{lemma}\label{(u_1 ot epsilon)* sigma } For all $h,k\in H$,
$$
(h^{(1)}\xcdot 1_A)\varsigma(h^{(2)}\ot k)=\varsigma(h^{(1)}\ot k)(h^{(2)}\xcdot 1_A) = \varsigma(h\ot k).
$$
\end{lemma}

\begin{proof} By Remark~\ref{pepe'} and item~(2),
\begin{align*}
&(h^{(1)}\xcdot 1_A)\varsigma(h^{(2)}\ot k) = (h^{(1)}\xcdot 1_A)(h^{(2)}\xcdot (k^{(1)}\xcdot 1_A)\varsigma(h^{(3)}\ot k^{(2)})\\
&\phantom{(h^{(1)}\xcdot 1_A)\varsigma(h^{(2)}\ot k)}= (h^{(1)}\xcdot (k^{(1)}\xcdot 1_A)\varsigma(h^{(2)}\ot k^{(2)})\\
&\phantom{(h^{(1)}\xcdot 1_A)\varsigma(h^{(2)}\ot k)}= \varsigma(h\ot k),
\shortintertext{and by items~(2), (11) and~(19),}
&\varsigma(h^{(1)}\ot k)(h^{(2)}\xcdot 1_A)= \varsigma(h^{(1)}\ot k^{(1)})(h^{(2)}k^{(2)}\xcdot 1_A)(h^{(3)}\xcdot 1_A)\\
&\phantom{\varsigma(h^{(1)}\ot k)(h^{(2)}\xcdot 1_A)}=\varsigma(h^{(1)}\ot k^{(1)})(h^{(2)}\xcdot (k^{(2)}\xcdot 1_A))(h^{(3)}\xcdot 1_A)\\
&\phantom{\varsigma(h^{(1)}\ot k)(h^{(2)}\xcdot 1_A)}= \varsigma(h^{(1)}\ot k^{(1)})(h^{(2)}\xcdot (k^{(2)}\xcdot 1_A))\\
&\phantom{\varsigma(h^{(1)}\ot k)(h^{(2)}\xcdot 1_A)}= \varsigma(h^{(1)}\ot k^{(1)})(h^{(2)}k^{(2)}\xcdot 1_A)\\
&\phantom{\varsigma(h^{(1)}\ot k)(h^{(2)}\xcdot 1_A)}= \varsigma(h\ot k),
\end{align*}
which finishes the proof.
\end{proof}

\begin{remark}\label{para que HL pase} By Remark~\ref{pepe'}
$$
\epsilon(h^{(1)}l^{(1)})\varsigma(h^{(2)}\ot l^{(2)})\! =\!  \epsilon(h^{(1)}l^{(1)}) (h^{(2)}l^{(2)}\xcdot 1_A) \varsigma(h^{(3)}\ot l^{(3)})\!=\!  (h^{(1)}l^{(1)}\xcdot 1_A) \varsigma(h^{(2)}\ot l^{(2)})\!=\!\varsigma(h\ot l).
$$
\end{remark}

\begin{proof}[Proof of Theorem~\ref{espanholes implica Bhom y Brzezinksi}] By item~(19) and \cite{GGV}*{Propositions~2.4 and 2.7}, the map $\sigma$ exists. Items~(2) and (4) are trivially fulfilled. By items~(17) and~(18), items~(8) and~(9) hold. Moreover by~\cite{GGV}*{Proposition~4.2} and items~(2), (4) and~(11), item~(1) also hold.  Let $\ov{\Pi}^L\coloneqq S^{-1}\circ \Pi^L$ and $l\in H^{\!L}$, $h\in H$ and $a\in A$. By~items~(2), (4) and~(11), Lemma~\ref{condicion'} and \cite{GGV}*{Propositions~4.1 and~4.2}, we have
\begin{align*}
&(l\xcdot 1_A)(h\xcdot a)= l\xcdot (h\xcdot a)=lh\xcdot a
\shortintertext{and}
& (h\xcdot a)(l\xcdot 1_A) = \ov{\Pi}^L(l)\xcdot (h\xcdot a) = S^{-1}(l)\xcdot (h\xcdot a)=S^{-1}(l)h\xcdot a,
\end{align*}
which proves item~(3). In order to check item~(5), it is sufficient to note that, by Lemma~\ref{(u_1 ot epsilon)* sigma }, \cite{BNS1}*{Equal\-ities~(2.7a) and~(2.7b)}, and item~(3),
\begin{align*}
& \sigma(lh, k)= (l^{(1)}h^{(1)}\xcdot 1_A)\sigma(l^{(2)}h^{(2)},k)\\
&\phantom{\sigma(lh, k)} = (lh^{(1)}\xcdot 1_A)\sigma(h^{(2)},k)\\
&\phantom{\sigma(lh, k)} = (l\xcdot 1_A)(h^{(1)}\xcdot 1_A)\sigma(h^{(2)},k)\\
&\phantom{\sigma(lh, k)}= (l\xcdot 1_A)\sigma(h, k),
\shortintertext{and}
& \sigma(S^{-1}(l)h, k)= \sigma(S^{-1}(l)^{(1)}h^{(1)}, k)(S^{-1}(l)^{(2)}h^{(2)}\xcdot 1_A)\\
&\phantom{\sigma(S^{-1}(l)h, k)}= \sigma(h^{(1)}, k)(S^{-1}(l)h^{(2)}\xcdot 1_A)\\
&\phantom{\sigma(S^{-1}(l)h, k)}= \sigma(h^{(1)}, k) (h^{(2)}\xcdot 1_A)(l\xcdot 1_A)\\
&\phantom{\sigma(S^{-1}(l)h, k)}=\sigma(h, k)(l\xcdot 1_A).
\end{align*}
By item~(11), \cite{BNS1}*{Equality~(2.7a)}, Lemma~\ref{(u_1 ot epsilon)* sigma } and~\cite{GGV}*{Proposition~2.8} (which applies by Remark~\ref{para que HL pase}), we have
$$
(h^{(1)}\xcdot (l\xcdot 1_A))\sigma(h^{(2)}, k) = (h^{(1)}l^{(1)}\xcdot 1_A)\sigma(h^{(2)}l^{(2)}, k)= \sigma(hl, k)= \sigma(h,lk),
$$
for all $l\in H^{\!L}$ and $h,k\in H$. So, item~(6) holds. Finally, item~(7)  follows from Remark~\ref{pepe} and items~(20) and~(21).
\end{proof}

\begin{proposition}\label{varsigma inversible implica sigma inversible} Under the hypotheses of Theorem~\ref{espanholes implica Bhom y Brzezinksi}, if $\varsigma$ is invertible, then $\sigma$ is also.
\end{proposition}

\begin{proof} Let $\bar{\varsigma}$ be the inverse of $\varsigma$. By~\cite{GGV}*{Proposition~5.13}, the map $\bar{\varsigma}$ factorizes  throughout a map $\bar{\sigma}\colon H\ot_{H^{\!L}} H\to A$. Item~(13) follows immediately from~item~(23), while item~(16) follows from items~(11) and~(24). We next prove item~(14). By item~(11) and Remark~\ref{rere1q}, we have
\begin{equation}\label{pepi}
\bar{\sigma}(h,k)= \bar{\sigma}(h^{(1)},k^{(1)})(h^{(2)}k^{(2)}\xcdot 1_A)\quad \text{for all $h,k\in H$.}
\end{equation}
Using this, items~(3), (13) and \cite{BNS1}*{Equalities~(2.7a) and~(2.7b)}, we obtain that
\allowdisplaybreaks
\begin{align*}
& \bar{\sigma}(lh,k)= (l^{(1)}h^{(1)}k^{(1)}\xcdot 1_A) \bar{\sigma}(l^{(2)}h^{(2)}, k^{(2)})\\
& \phantom{\bar{\sigma}(lh,k)} = (l h^{(1)}k^{(1)}\xcdot 1_A) \bar{\sigma}(h^{(2)},k^{(2)}) \\
&\phantom{\bar{\sigma}(lh,k)}= (l\xcdot 1_A)( h^{(1)}k^{(1)}\xcdot 1_A) \bar{\sigma}(h^{(2)}, k^{(2)})\\
&\phantom{\bar{\sigma}(lh,k)}= (l\xcdot 1_A)\bar{\sigma}(h,k)
\shortintertext{and}
& \bar{\sigma}(S^{-1}(l)h,k)= \bar{\sigma}(S^{-1}(l)^{(1)}h^{(1)}, k^{(1)}) (S^{-1}(l)^{(2)}h^{(2)}k^{(2)}\xcdot 1_A)\\
& \phantom{\bar{\sigma}(S^{-1}(l)h,k)} = \bar{\sigma}(h^{(1)},k^{(1)}) (S^{-1}(l)h^{(2)}k^{(2)}\xcdot 1_A)\\
& \phantom{\bar{\sigma}(S^{-1}(l)h,k)} = \bar{\sigma}(h^{(1)}, k^{(1)}) (h^{(2)}k^{(2)}\xcdot 1_A)(l\xcdot 1_A)\\
& \phantom{\bar{\sigma}(S^{-1}(l)h,k)} = \bar{\sigma}(h,k) (l\xcdot 1_A)
\end{align*}
for all $h,k\in H$ and $l\in H^{\!L}$, which is item~(14). It remains to prove item~(15). By the first equality in~\eqref{pepi}, \cite{BNS1}*{Equality~(2.7b)} and items~(2), (3) and~(11), we obtain that
\begin{align*}
& \bar{\sigma}(h, S^{-1}(l)k) = \bar{\sigma}(h^{(1)}, S^{-1}(l)^{(1)}k^{(1)})(h^{(2)}S^{-1}(l)^{(2)}k^{(2)}\xcdot 1_A)\\
& \phantom{\bar{\sigma}(h, S^{-1}(l)k)} = \bar{\sigma}(h^{(1)}, k^{(1)})(h^{(2)}S^{-1}(l)k^{(2)}\xcdot 1_A)\\
& \phantom{\bar{\sigma}(h, S^{-1}(l)k)} = \bar{\sigma}(h^{(1)}, k^{(1)})(h^{(2)}\xcdot (S^{-1}(l)k^{(2)}\xcdot 1_A))\\
& \phantom{\bar{\sigma}(h, S^{-1}(l)k)} = \bar{\sigma}(h^{(1)}, k^{(1)})(h^{(2)}\xcdot ((k^{(2)}\xcdot 1_A)(l\xcdot 1_A)))\\
& \phantom{\bar{\sigma}(h, S^{-1}(l)k)} = \bar{\sigma}(h^{(1)}, k^{(1)})(h^{(2)}\xcdot (k^{(2)}\xcdot 1_A))(h^{(3)}\xcdot (l\xcdot 1_A))\\
& \phantom{\bar{\sigma}(h, S^{-1}(l)k)} = \bar{\sigma}(h^{(1)}, k^{(1)})(h^{(2)} k^{(2)}\xcdot 1_A)(h^{(3)}\xcdot (l\xcdot 1_A))\\
& \phantom{\bar{\sigma}(h, S^{-1}(l)k)} = \bar{\sigma}(h^{(1)}, k)(h^{(2)}\xcdot (l\xcdot 1_A))
\end{align*}
for all $h,k\in H$ and $l\in H^{\!L}$, as desired.
\end{proof}

\subsection{The isomorphism}
Let $H$ be a weak Hopf algebra, let $A$ be an algebra, let $p\colon H\ot H\to H\ot_{H^{\!R}} H$ be the canonical surjection and let $\rho\colon H\ot A\to A$ and $\sigma\colon H\ot_{H^{\!R}} H\to A$ be maps. Assume that $\rho$ and $\sigma$ satisfy conditions~(1)--(10) (from which $\rho$ and $\varsigma\coloneqq \sigma\xcirc p$  satisfy conditions~~(2), (4), (11) and~(17)--(22)).

\smallskip

Let $i\colon A\times^{\rho}_{\varsigma} H\to  A\ot H$ be the canonical inclusion and let $\pi\colon A\ot H\to A\#^{\rho}_{\sigma} H$ be the canonical surjection. Let $\nabla_{\!\rho}$ be as above of Theorem~\ref{teor: producto cruzado espanhol}. Since, for $a\in A$ and $h\in H$,
$$
(\pi\xcirc \nabla_{\!\rho})(a\ot h)\hsa =\hsa  a(h^{(\hsb 1\hsb )}\xcdot 1_A)\ot_{H^{\!L}}\hsa h^{(\hsb 2\hsb )}\hsa =\hsa  a\xcdot \Pi^L(h^{(\hsb 1\hsb )})\ot_{H^{\!L}} h^{(\hsb 2\hsb )}\hsa =\hsa a\ot_{H^{\!L}} \Pi^L(h^{(\hsb 1\hsb )})h^{(\hsb 2\hsb )}\hsa =\hsa a\ot_{H^{\!L}}\hsa h,
$$
we have $\pi\xcirc \nabla_{\!\rho} = \pi$. Clearly we also have $\nabla_{\!\rho}\xcirc i=i$.

\begin{proposition}\label{isomorfismo}
The map $\psi\colon A\times^{\rho}_{\varsigma} H \to A\#^{\rho}_{\sigma} H$, defined by $\psi\coloneqq \pi\xcirc i$, is a left $A$-linear and right $H$-colinear algebra isomorphism.
\end{proposition}

\begin{proof} Let $a,b\in A$ and $h,k\in H$. Since
\allowdisplaybreaks
\begin{align*}
\psi(a\times h)\psi(b\times k) & = (a(h^{(1)}\xcdot 1_A)\ot_{H^{\!L}} h^{(2)})(b(k^{(1)}\xcdot 1_A)\ot_{H^{\!L}} k^{(2)})\\
& = a (h^{(1)}\xcdot 1_A)\bigl(h^{(2)}\xcdot (b(k^{(1)}\xcdot 1_A))\bigr)\sigma(h^{(3)},k^{(2)}) \ot_{H^{\!L}} h^{(4)}k^{(3)}\\
& = a \bigl(h^{(1)}\xcdot (b(k^{(1)}\xcdot 1_A))\bigr)\sigma(h^{(2)},k^{(2)}) \ot_{H^{\!L}} h^{(3)}k^{(3)}\\
& = a (h^{(1)}\cdot b)(h^{(2)}\xcdot (k^{(1)}\xcdot 1_A)) \sigma(h^{(3)},k^{(2)})\ot_{H^{\!L}} h^{(4)}k^{(3)}\\
& = a (h^{(1)}\cdot b)\sigma(h^{(2)},k^{(1)})(h^{(3)}k^{(2)}\xcdot 1_A)\ot_{H^{\!L}} h^{(4)}k^{(3)}\\
&= \psi\bigl(a (h^{(1)}\cdot b)\varsigma(h^{(2)}\ot k^{(1)}) \times h^{(3)}k^{(2)}\bigr)\\
& = \psi\bigl((a\times h)(b\times k)\bigr).
\end{align*}
and
$$
\psi(1_A\times 1)\! =\! 1^{(1)}\xcdot 1_A\ot_{H^{\!L}} 1^{(2)}\! =\! (1^{(1)}\xcdot 1_A)\cdot 1^{(2)}\ot_{H^{\!L}} 1 = (1^{(1)}\xcdot 1_A)(1^{(2)}\xcdot 1_A) \ot_{H^{\!L}} 1\! =\! 1_A \ot_{H^{\!L}} 1,
$$
the map $\psi$ is a unitary algebra morphism. It is clear that it is also left $A$-linear and right $H$-colinear. Let $l\in H^{\!L}$, $h\in H$ and $a\in A$. Since, by item~(3) and~\cite{BNS1}*{Equality~(2.7a)},
$$
\nabla_{\!\rho}(a\xcdot l\ot h)\! =\!  a(l\cdot 1_A)(h^{(\hsb 1\hsb )}\xcdot 1_A)\ot h^{(\hsb 2\hsb )}\!  =\!  a(lh^{(\hsb 1\hsb )}\xcdot 1_A)\ot h^{(\hsb 2\hsb )}\!=\!  a(l^{(\hsb 1\hsb )}h^{(\hsb 1 \hsb )}\xcdot 1_A)\ot l^{(\hsb 2 \hsb )}h^{(\hsb 2 \hsb )}\! =\! \nabla_{\!\rho}(a\ot lh),
$$
there exists $\phi\colon A\#^{\rho}_{\sigma} H\to A\times^{\rho}_{\varsigma} H$ such that $i\xcirc \phi\xcirc \pi=\nabla_{\!\rho}$. The computations
$$
i\xcirc \phi\xcirc \psi= i\xcirc \phi\xcirc \pi \xcirc i=\nabla_{\!\rho}\xcirc i=i\qquad\text{and}\qquad \psi\xcirc \phi\xcirc \pi = \pi\xcirc i\xcirc \phi\xcirc \pi=\pi\xcirc \nabla_{\!\rho}= \pi,
$$
prove that $\phi$ and $\psi$ are inverses one of each other.
\end{proof}

\begin{remark} The isomorphism introduced in the previous proposition is natural in an evident sense.
\end{remark}

\begin{example} Let $k$ be a field and $K\coloneqq k\times k$. Let $C\coloneqq \{1,g\}$ be the cyclic order $2$ and let $H\coloneqq K\ot k[C]\ot K$. For $a,b,c,d\in k$, we set $\lambda_{ab}^{cd}\coloneqq (a,b)\ot 1\ot (c,d)$ and $G_{ab}^{cd}\coloneqq (a,b)\ot g\ot (c,d)$. The $k$-vector space $H$ is a weak Hopf algebra via the multiplication defined by
$$
\lambda_{ab}^{cd}\lambda_{a'b'}^{c'd'}\coloneqq \lambda_{aa',bb'}^{cc',dd'},\quad \lambda_{ab}^{cd}G_{a'b'}^{c'd'}\coloneqq G_{aa',bb'}^{dc',cd'},\quad G_{ab}^{cd}\lambda_{a'b'}^{c'd'}\coloneqq G_{ab',ba'}^{cc',dd'}\quad\text{and}\quad G_{ab}^{cd}G_{a'b'}^{c'd'}\coloneqq \lambda_{ab',ba'}^{dc',cd'},
$$
and the comultiplication defined by
$$
\Delta(\lambda_{ab}^{cd})\coloneqq \lambda_{ab}^{10}\ot \lambda_{10}^{cd} + \lambda_{ab}^{01}\ot \lambda_{01}^{cd}\quad\text{and}\quad \Delta(G_{ab}^{cd})\coloneqq G_{ab}^{10}\ot G_{01}^{cd} + G_{ab}^{01}\ot G_{10}^{cd}.
$$
The unit is $\lambda_{11}^{11}$, the counit is the map $\epsilon\colon H\to k$ given by $\epsilon(\lambda_{ab}^{cd}) = ac+bd$ and $\epsilon(G_{ab}^{cd}) = ad+bc$, and the antipode is the map $S\colon H\to H$ given by $S(\lambda_{ab}^{cd}) = \lambda_{cd}^{ab}$ and $S(G_{ab}^{cd}) = G_{cd}^{ab}$. A straightforward computation shows that the maps $\Pi^{\!L}\colon H\to H$ and $\Pi^{\!R}\colon H\to H$ are given by
$$
\Pi^{\!L}(\lambda_{ab}^{cd}) = \lambda_{ac,bd}^{11},\quad \Pi^{\!L}(G_{ab}^{cd}) = \lambda_{ad,bc}^{11},\quad \Pi^{\!R}(\lambda_{ab}^{cd}) = \lambda_{11}^{ac,bd}\quad\text{and}\quad \Pi^{\!R}(G_{ab}^{cd}) = \lambda_{11}^{bc,ad}.
$$
The map $\rho\colon H\ot K\to K$, defined by
$$
\lambda_{ab}^{cd}\cdot (x,y)\coloneqq (axc,byd)\quad\text{and}\quad G_{ab}^{cd}\cdot (x,y) \coloneqq (axd,byc),
$$
where $h\cdot (x,y)$ denotes $\rho(h)(x,y)$ satisfies conditions~(1)--(4); and the map $\sigma\colon H\ot_{H^{\!R}} H\to K$, defined by
\begin{alignat*}{2}
&\sigma(\lambda_{ab}^{cd},\lambda_{a'b'}^{c'd'})\coloneqq (aa'cc',bb'dd'),&&\qquad \sigma(\lambda_{ab}^{cd},G_{a'b'}^{c'd'})\coloneqq   (aa'cd',bb'dc'),\\
&\sigma(G_{ab}^{cd},\lambda_{a'b'}^{c'd'}) \coloneqq  (aa'dd',bb'cc'),&&\qquad \sigma(G_{ab}^{cd},G_{a'b'}^{c'd'}) \coloneqq  (aa'dc',bb'cd'),
\end{alignat*}
satisfies conditions~(5)--(9). Moreover $\sigma$ is invertible in the sense of Definition~\ref{cociclo inversible}. Since $\sigma$ does not fulfill condition~(10), this shows that condition~(10) in Theorem~\ref{Bhom y Brzezinski son espanholes} can not be eliminated.
\end{example}

\end{document}